\documentclass[reqno]{amsart}
\usepackage{amsmath, amssymb, amsthm, epsfig, enumitem}
\usepackage{hyperref, latexsym}
\usepackage{url}
\usepackage[mathscr]{euscript}

\usepackage{color}
\usepackage{fullpage} 
\usepackage{setspace}

\def\today{\ifcase\month\or
  January\or February\or March\or April\or May\or June\or=
  July\or August\or September\or October\or November\or December\fi
  \space\number\day, \number\year}

 \newtheorem{theorem}{Theorem}
  
 \newtheorem{lemma}[theorem]{Lemma}
 \newtheorem{proposition}[theorem]{Proposition}
 
 \theoremstyle{definition}

 \theoremstyle{remark}

\newcommand{\intav}[1]{\mathchoice {\mathop{\vrule width 6pt height 3 pt depth  -2.5pt
\kern -8pt \intop}\nolimits_{\kern -6pt#1}} {\mathop{\vrule width
5pt height 3  pt depth -2.6pt \kern -6pt \intop}\nolimits_{#1}}
{\mathop{\vrule width 5pt height 3 pt depth -2.6pt \kern -6pt
\intop}\nolimits_{#1}} {\mathop{\vrule width 5pt height 3 pt depth
-2.6pt \kern -6pt \intop}\nolimits_{#1}}}

\begin{document}

\title[Continuity for the one-dimensional centered Hardy-Littlewood maximal operator at the derivative level]{Continuity for the one-dimensional centered Hardy-Littlewood maximal operator at the derivative level}
\author[Gonz\'{a}lez-Riquelme ]{Cristian Gonz\'{a}lez-Riquelme}
\address{IMPA - Instituto de Matem\'{a}tica Pura e Aplicada\\
Rio de Janeiro - RJ, Brazil, 22460-320.}
\email{cristian@impa.br}

\allowdisplaybreaks
\numberwithin{equation}{section}

\maketitle

\begin{abstract}
We prove the continuity of the map
$f \mapsto (Mf)'$
from $W^{1,1}(\mathbb{R})$ to $L^1(\mathbb{R})$, where $M$ is the centered Hardy-Littlewood maximal operator. This solves a question posed by Carneiro, Madrid and Pierce.
\end{abstract}
\section{Introduction}
Maximal operators are central objects in analysis.  The most classical of these operators is the centered Hardy-Littlewood maximal operator. This is defined as follows: for any $f\in L^{1}_{\text{loc}}(\mathbb{R}^d)$ and $x\in \mathbb{R}^d$ 
\begin{align*}
Mf(x):=\sup_{r>0}\frac{\int_{B(x,r)}|f|}{|B(x,r)|}=:\sup_{r>0}\,\intav{B(x,r)}|f|,\end{align*} where $|X|$ is the Lebesgue measure of the measurable set $X\subset \mathbb{R}^d$. We define $\widetilde{M}$ as its uncentered version, where the supremum is taken over all balls that contain $x$ but are not necessarily centered at $x$. 
The regularity theory for these operators started  with Kinnunen \cite{Kinnunen1997}, who proved that \begin{align}\label{hardylittlewoodmapluiro}
f\mapsto Mf\end{align} is bounded from $W^{1,p}(\mathbb{R}^d)$ to itself when $p>1$. The same result for $\widetilde{M}$ follows by similar methods. The case $p=1$ is much more delicate. Certainly, since $Mf\notin L^{1}(\mathbb{R}^d)$ for any non-trivial $f$, one cannot expect Kinnunen's result to hold for $p=1$. Nonetheless, the boundedness is true at the derivative level when $p=1$
and $n = 1$, i.e. the map $f\mapsto (Mf)'$ is bounded from $W^{1,1}(\mathbb{R})$ to $L^{1}(\mathbb{R}).$ This was established by Kurka in \cite{Kurka2010}, while the same result for $\widetilde{M}$ was obtained by Tanaka \cite{Tanaka2002} and sharpened by Aldaz and P\'{e}rez-L\'{a}zaro \cite{AP2007}. 
This boundedness has been also investigated in higher dimensions. Luiro \cite{Luiro2018} established the boundedness of the map $f\mapsto |\nabla \widetilde{M}f|$ from $W^{1,1}_{\text{rad}}(\mathbb{R}^d)$ to $L^{1}(\mathbb{R}^d)$, while Weigt \cite{weigtcharacteristic} proved the boundedness when restricting the map to characteristic functions.

The continuity for these types of maps has been an active topic of research over the last years. According to \cite[Question 3]{HO2004}, it was asked by T. Iwaniec whether the map \eqref{hardylittlewoodmapluiro} is continuous from $W^{1,p}(\mathbb{R}^d)$ to itself, when $p>1$. This question was answered affirmatively by Luiro in \cite{Luiro2007}. Again, the endpoint case $p=1$ is significantly more involved. For the uncentered Hardy-Littlewood maximal operator, the continuity of the map
\begin{align*}
f\mapsto \left(\widetilde{M}f\right)'
\end{align*}
from $W^{1,1}(\mathbb{R})$ to $L^{1}(\mathbb{R})$ was proved by Carneiro, Madrid and Pierce in \cite{CMP2017}. This was later generalized in \cite{GRK} and \cite{CGRM} to the $BV(\mathbb{R})$ case and to the higher dimensional radial case, respectively. In the fractional version of this problem, based on previous developments made in \cite{BM2019,Madrid2017,weigt2020sobolev}, the general case was obtained in \cite{beltran2021continuity}. However, the centered classical case is beyond the scope of the methods developed in these previous works. 

Another family of operators was also object of study in this topic: maximal operators of convolution type associated to smooth kernels. We write $M_{\phi}$ for the centered maximal operator associated to a radially non-increasing kernel $\phi\in L^{1}(\mathbb{R}^d)$. The boundedness of the map $f\mapsto (M_{\phi}f)'$ from $W^{1,1}(\mathbb{R})$ to $L^{1}(\mathbb{R})$ was proved in \cite{CS2013} and later in \cite{CFS2015} for certain smooth kernels $\phi$ related to partial differential equations that include the heat and Poisson kernels. A radial version of these results was achieved in \cite{CGR}. A step forward for the understanding of the continuity in the centered setting was made recently by the author in the work \cite{gonzalezriquelme2021continuity}, where the continuity of some of these maps was established.

In the present manuscript, we establish the continuity for the centered Hardy-Littlewood maximal operator, solving a question posed by Carneiro, Madrid and Pierce in \cite[Question A]{CMP2017} and establishing, in the one-dimensional case, the endpoint version of \cite[Question 3]{HO2004} at the derivative level.
\begin{theorem}\label{theoremcentered}
We have that the map $$f\mapsto (Mf)'$$
is continuous from $W^{1,1}(\mathbb{R})$ to $L^{1}(\mathbb{R}).$
\end{theorem}
We notice that the map considered here is well defined and bounded  (see Lemma \ref{kurkaslemma}). 
We highlight that the methods developed in the aforementioned works \cite{CGRM,CMP2017,gonzalezriquelme2021continuity,GRK} are not enough to conclude our result. For instance, in the works \cite{CGRM, CMP2017,GRK} it is important that the operator $\widetilde{M}$ has the flatness property; this is, that the maximal functions have a.e. zero derivative at the points where they coincide with the original function.  In \cite{gonzalezriquelme2021continuity}, the subharmonicity property, which the maximal functions considered there satisfy, plays a crucial role in the proof of the continuity. The centered Hardy-Littlewood maximal operator does not satisfy either of these properties, therefore, new insights are required in order to achieve our result. Our method is based on a decomposition of $M$ as a maximum of two operators $M_1$ and $M_2,$ both of them depending on $f$ and on a simple function $g_{\varepsilon}$ that approximates $f'$ in $L^{1}(\mathbb{R})$. The operator $M_1,$ the local one, is restricted to balls that are contained in the support of an interval determined by $g_{\varepsilon}$. On the other hand, the operator $M_2$, the global one, is restricted to balls that are not contained in any of these lines. The idea is that, since the operator $M_1$ is well behaved with respect to some lines, it is possible to conclude that $M_1f_j$ is close to $f_j$ at the derivative level, for any $j$ big enough. A different approach is needed in order to deal with the contribution of the operator $M_2$, for this we shall take advantage of the fact that the radii considered in $M_2$ are generally bounded by below. In essence, this yields a smoother nature to this operator that is helpful for our purposes.

Considering the progress made in this manuscript, we summarize the situation of the {\it endpoint continuity program} (originally proposed in \cite[Table 1]{CMP2017}) in the table below. The word YES in a box means that the continuity of the corresponding map has been proved, whereas the word NO means that it has been shown that it fails. We notice that after this work the only open problem in this program is to determine if the map $f\mapsto Mf$ is continuous from $BV(\mathbb{R})$ to itself, marked with OPEN in the table below.
\begin{table}[h]
\renewcommand{\arraystretch}{1.3}
\centering
\caption{Endpoint continuity program}
\label{Table-ECP}
\begin{tabular}{|c|c|c|c|c|}
\hline
 \raisebox{-1.3\height}{------------}&  \parbox[t]{2.8cm}{  $W^{1,1}-$continuity; \\ continuous setting} &  \parbox[t]{2.8cm}{  $BV-$continuity; \\ continuous setting} & \parbox[t]{2.6cm}{  $W^{1,1}-$continuity; \\ discrete setting} &  \parbox[t]{2.5cm}{  $BV-$continuity; \\ discrete setting} \\ [0.5cm]
\hline
 \parbox[t]{3.3cm}{ Centered classical \\ maximal operator} &  \raisebox{-0.8\height}{YES: Theorem \ref{theoremcentered}} & \raisebox{-0.8\height}{OPEN} & \raisebox{-0.8\height}{YES$^2$}  &\raisebox{-0.8\height}{YES$^4$}\\[0.5cm]
 \hline
  \parbox[t]{3.3cm}{ Uncentered classical \\ maximal operator} &  \raisebox{-0.8\height}{YES$^1$} &  \raisebox{-0.8\height}{YES$^6$} &  \raisebox{-0.8\height}{YES$^2$} &  \raisebox{-0.8\height}{YES$^1$} \\[0.5cm]
 \hline
 \parbox[t]{3.3cm}{ Centered fractional \\ maximal operator} &  \raisebox{-0.8\height}{YES$^5$} &  \raisebox{-0.8\height}{NO$^1$} &  \raisebox{-0.8\height}{YES$^3$} & \raisebox{-0.8\height}{NO$^1$}  \\[0.5cm]
 \hline
\parbox[t]{3.3cm}{ Uncentered fractional \\ maximal operator} &  \raisebox{-0.8\height}{YES$^4$} &  \raisebox{-0.8\height}{NO$^1$} &  \raisebox{-0.8\height}{YES$^3$}  & \raisebox{-0.8\height}{NO$^1$} \\[0.5cm]
 \hline
\end{tabular}
\vspace{0.05cm}
\flushleft{
\ \ \footnotesize{$^1$ Result previously obtained in \cite{CMP2017}.\\
\ \ $^2$ Result previously obtained in \cite[Theorem 1]{CH2012}. \\
\ \ $^3$ Result previously obtained in \cite[Theorem 3]{CM2015}.\\
\ \ $^4$ Result previously obtained in \cite{Madrid2017}.\\
\ \ $^5$ Result previously obtained in \cite{BM2019.2}.\\
\ \ $^6$ Result previously obtained in \cite{GRK}
}}
\end{table}
\section{Preliminaries}
 In this section we discuss some preliminary results for our purposes. Let us consider $f_j\to f$ in $W^{1,1}(\mathbb{R})$. In order to prove Theorem \ref{theoremcentered}, by \cite[Lemma 14]{CMP2017} we may assume henceforth that $f_j,f\ge 0$. Also, since the case $f=0$ of Theorem \ref{theoremcentered} follows by the boundedness, we assume that $f\neq 0$.
We start with the well known Luiro's formula.
\begin{proposition}[Case $p=1$ of {\cite[Theorem 3.1]{Luiro2007}}]\label{Luirosformulacentered} Let us take $g\in W^{1,1}(\mathbb{R})$. Assume that $Mg$ is differentiable at the point $x$, if $Mg(x)=\intav{[x-r,x+r]}|g|$ with $r>0$, we have that
$$(Mg)'(x)=\intav{[x-r,x+r]}|g|'.$$
\end{proposition}
\begin{proof}
This follows from \cite[Proposition 2.4]{beltran2021continuity} and the remark thereafter.
\end{proof}
The next result provide us with a local control for the variation of $M$. For any interval $I$ (not necessarily finite)  and $g\in L^{1}(\mathbb{R})$ we define $$M_{I}g(x):=\sup_{[x-r,x+r]\subset I}\intav{[x-r,x+r]}|g|.$$ 
\begin{lemma}\label{kurkaslemma}
If $f\in W^{1,1}(I)$, we have that $M_{I}f$ is absolutely continuous and that there exists an universal constant $C$, such that $$\int_{I}|(M_{I}f)'|\le C\int_{I}|f'|.$$
\end{lemma}
\begin{proof}
The absolutely continuity of $M_{I}f$ can be concluded by following the reasoning in \cite[Corollary 1.3]{Kurka2010}. The boundedness follows from \cite[Remark 6.4]{Kurka2010}.
\end{proof}
Also, we need the following uniform control near a finite number of points. 
\begin{lemma}\label{controloverfinitenumerofpoints}
Let $f_j\to f$ in $W^{1,1}(\mathbb{R})$. Let $\{p_1,\dots,p_s\}$ be a finite set. For any $\varepsilon>0$ there exists $\delta>0$ such that, for $j$ big enough, we have
$$\sum_{i=1}^{s}\int_{[p_i-\delta,p_i+\delta]}|(Mf_j)'|<\varepsilon$$.
\end{lemma}
\begin{proof}
This proof follows a similar path than the one presented originally in \cite[Proposition 19]{GR}.
It is enough to prove that there exists $\delta>0$ such that 
\begin{align*}
\int_{[p_i-\delta,p_i+\delta]}|(Mf_j)'|<\frac{\varepsilon}{s}
\end{align*}
for any fixed $i$ and $j$ big enough.  Let us take $\delta_i>0$ such that $$\int_{(a_i-\delta_i,a_i+\delta_i)}|f'|<\frac{\varepsilon}{2Cs},$$ where $C$ is the universal constant that appears in Lemma \ref{kurkaslemma}. For $j$ big enough we have $$\int_{(a_i-\delta_i,a_i+\delta_i)}|f_j'|<\frac{\varepsilon}{2Cs}.$$ For any given $\ell\in \mathbb{Z}_{>0}$ let us define $$A^1_{\ell,j}:=\left\{x\in \left(a_i-\frac{\delta_i}{\ell},a_i+\frac{\delta_i}{\ell}\right); Mf_j(x)=M_{(a_i-\delta_i,a_i+\delta_i)}f_j(x)\right\}$$ and $$A^2_{\ell,j}=\left\{x\in \left(a_i-\frac{\delta_i}{\ell},a_i+\frac{\delta_i}{\ell}\right); Mf_j(x)>M_{(a_i-\delta_i,a_i+\delta_i)}f_j(x)\right\}.$$ Since $Mf_j\ge M_{(a_i-\delta_i,a_i+\delta_i)}f_j$ know that $(Mf_j)'=(M_{(a_i-\delta_i,a_i+\delta_i)}f_j)'$ a.e. in  $A^1_{\ell,j}.$ Therefore 
\begin{align*}
\int_{A^{1}_{\ell,j}}|(Mf_j)'|\le \int_{(a_i-\delta_i,a_i+\delta_i)}|(M_{(a_i-\delta_i,a_i+\delta_i)}f_j)'|\le C\int_{(a_i-\delta_i,a_i+\delta_i)}|f_j'|\le \frac{\varepsilon}{2s}.
\end{align*}
Also, for a.e. $x\in A^2_{\ell,j}$, we have that there exists $r_x\ge \delta_i-\frac{\delta_i}{\ell}=\frac{\delta_i(\ell-1)}{\ell}$ such that $\intav{[x-r_{j,x},x+r_{j,x}]}f_j=Mf_j(x)$. Then, by Luiro's formula (Proposition \ref{Luirosformulacentered}), we have that $(Mf_j)'(x)=\intav{[x-r_{j,x},x+r_{j,x}]}f_j',$ and therefore $|(Mf_j)'(x)|\le \frac{1}{2r_{j,x}}\|f_j'\|_1.$ Thus, for $x\in A^2_{\ell,j}$ we have $$|(Mf_j)'(x)|\le \frac{\delta_i \ell}{2(\ell-1)}\|f_j'\|_1.$$ In consequence, we have
\begin{align*}\int_{A^{2}_{\ell,j}}|(Mf_j)'|\le \int_{(a_i-\frac{\delta_i}{\ell},a_i+\frac{\delta_i}{\ell})}\frac{\delta_i \ell}{2(\ell-1)}\|f_j'\|_1\le \frac{\delta_i^2}{(\ell-1)}\|f_j'\|_1. 
\end{align*}
From here, we conclude our lemma by choosing  $\ell$ such that $\frac{\delta_i^2}{\ell-1}< \frac{\varepsilon}{4s}$, $\delta:=\frac{\delta_i}{\ell}$ and by taking $j$ big enough such that  $\frac{\|f'\|_1}{2}\le \|f_j'\|_1\le  \frac{3\|f'\|_1}{2}$.
\end{proof} 
Also, we need the following uniform control near infinity.
\begin{lemma}[{\cite[Proposition 4.11]{BM2019}}]\label{controlnearinfinitycc} 
Let $f_j\to f$ in $W^{1,1}(\mathbb{R})$ and $\varepsilon>0$ be given. There exists $K>0$ such that, for $j$ big enough, we have   $$\int_{(-K,K)^{c}}|(Mf_j)'|< \varepsilon.$$
\end{lemma}

\section{The auxiliary maximal operators}
\begin{figure}[ht]
\includegraphics[scale = 0.8]{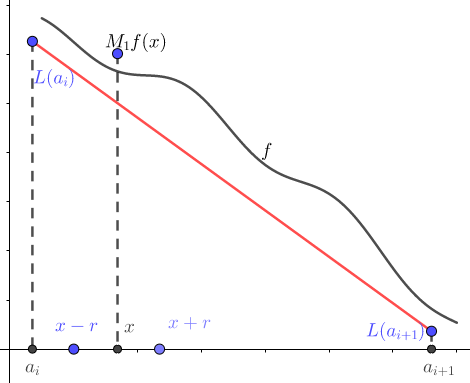}
\caption{In the figure the scope of $L$ is $\alpha_i$ and  $[x-r,x+r]$ is an admissible interval for $x$ and $M_1$.}
\label{stargraphextremizer}
\end{figure}
In this section we define the main objects of our work. Let us take $\varepsilon>0$ and consider $g_{\varepsilon}=\displaystyle \sum_{i=0}^{N}\alpha_{i}\chi_{(a_i,a_{i+1})}$ such that $\|f'-g_{\varepsilon}\|_{1}<\varepsilon$. That is, we approximate the derivative of our limit function by a simple function. We write $a_0=-\infty$, $a_{N+1}=\infty$ and $\mathcal{P}:=\{a_1,\dots, a_N\}$. We assume that $\mathcal{P}$ is non-empty. We observe that $\alpha_{0}=\alpha_{n}=0$. Now, we define our auxiliary maximal operators $M_1,M_2$ as follows: for any $h\in L^{1}(\mathbb{R})$ and $x\in \mathbb{R}$ we set $$M_1h(x):=\sup_{r<d(x,\mathcal{P})}\intav{[x-r,x+r]}|h|,$$
and $$M_2h(x):=\sup_{r\ge d(x,\mathcal{P})}\intav{[x-r,x+r]}|h|.$$
We now state some basic results about our operators $M_1,M_2.$
\begin{lemma}\label{uniformconvergencecentered}
Let $f_j\to f$ in $W^{1,1}(\mathbb{R}).$ We have $$M_if_j\to M_if$$
uniformly, for $i=1,2.$
\end{lemma}
\begin{proof}
It follows from the fact that $|M_{i}f_j-M_{i}f|\le |M_{i}(f_j-f)|\le \|f_j-f\|_{\infty}.$
\end{proof}
\subsection{Properties of $M_2$}For any $K,\delta>0$ such that the intervals $(a_i-\delta,a_i+\delta)$ are pairwise disjoint, let us define $U_{\delta,K}=(-K,K)\setminus \cup_{i=1}^{n} (a_i-\delta,a_i+\delta).$ We observe that for any $x\in U_{\delta,K}$ and any $g\in W^{1,1}(\mathbb{R})$ there exists a radius $r_x\ge \delta$ such that $\intav{[x-r_x,x+r_x]}|g|=M_2\,g(x).$ 
 We have then the following.
\begin{lemma}
For any $g\in W^{1,1}(\mathbb{R})$ we have that $M_2\,g$ is weakly differentiable in $U_{\delta,K}.$
\end{lemma}
\begin{proof}
For any $x,y\in U_{\delta,K}$ with $M_2\,g(x)>M_{2}\,g(y)$, we have \begin{align*}M_{2}\,g(x)-M_{2}\,g(y)=\intav{[x-r_x,x+r_x]}|g|-\intav{[y-r_y,y+r_y]}|g|&\le  \intav{[x-r_x,x+r_x]}|g|-\intav{[y+|x-y|-r_x,y+|x-y|+r_x]}|g|\\
&\le \|g\|_1\left(\frac{1}{2r_x}-\frac{1}{2r_x+2|x-y|}\right)\le \|g\|_1C(\delta)|x-y|,
\end{align*} 
where $C(\delta)$ is the Lipschitz constant of the function $\frac{1}{2x}$ in the set $[\delta,\infty).$ Therefore, we have that $M_2\,g$ is Lipschitz in this set, from where we conclude our lemma. 
\end{proof}
In the next result we present a formula for the derivative of $M_2\,g$ that has similarities with the one presented in \cite[Lemma 10]{GRK}. We use the notation $x\pm \infty=\pm \infty.$
\begin{lemma}\label{formuladerivativecnentered}
Let $g\in W^{1,1}(\mathbb{R})$. Let $x\in U_{\delta,K}$ be such that $M_2\,g$ is differentiable at $x$, and let $r_x$ such that $M_2\,g(x)=\intav{[x-r_x,x+r_x]}|g|$ with $r_x\ge d(x,\mathcal{P})$. Assume that $\frac{a_i+a_{i+1}}{2}<x<a_{i+1}.$ 
Then, we have $$(M_2\,g)'(x)=\frac{\int_{[x-r_x,x+r_x]}|g|}{2r_x^2}-\frac{|g|(x-r_x)}{r_x}.$$
\end{lemma}
\begin{proof}
Observe that, for $h>0$, we have
\begin{align*}
 \frac{M_2\,g(x)-M_2\,g(x-h)}{h}&\le \frac{\frac{\int_{[x-r_x,x+r_x]}|g|}{2r_x}-\frac{\int_{[x-r_x-2h,x+r_x]}|g|}{2r_x+2h}}{h}\\
 &=\frac{\frac{\int_{[x-r_x,x+r_x]}|g|}{2r_x}-\frac{\int_{[x-r_x,x+r_x]}|g|}{2r_x+2h}-\frac{\int_{[x-r_x-2h,x-r_x]}|g|}{2r_x+2h}}{h}\\
 &\to \frac{\int_{[x-r_x,x+r_x]}|g|}{2r_x^2}-\frac{|g|(x-r_x)}{r_x}
\end{align*}
when $h\to 0,$ where we use the continuity of $g$. Therefore $(M_2\,g)'(x)\le \frac{\int_{[x-r_x,x+r_x]}|g|}{2r_x^2}-\frac{|g|(x-r_x)}{2r_x}$. Also, for $h>0$, since $x<a_i\le x+r_x$ (and hence the interval $[x-r_x+2h,x+r_x]$ is admissible for $x+h$ for the operator $M_2$), we have
\begin{align*}
    \frac{M_2\,g(x+h)-M_2\,g(x)}{h}&\ge \frac{\frac{\int_{[x-r_x+2h,x+r_x]}|g|}{2r_x-2h}-\frac{\int_{[x-r_x,x+r_x]}|g|}{2r_x}}{h}\\
    &=\int_{[x-r_x,x+r_x]}|g|\left(\frac{\frac{1}{2r_x-2h}-\frac{1}{2r_x}}{h}\right)-\frac{\int_{[x-r_x,x-r_x+2h]}|g|}{(2r_x-2h)h}\\
    &\to \frac{\int_{[x-r_x,x+r_x]}|g|}{2r_x^2}-\frac{|g|(x-r_x)}{r_x}
\end{align*}
when $h\to 0$, and therefore $(M_2\,g)'(x)\ge \frac{\int_{[x-r_x,x+r_x]}|g|}{2r_x^2}-\frac{|g|(x-r_x)}{r_x},$ from where we conclude our lemma.
\end{proof}
\begin{lemma}\label{convergenceofradiuscentered}
Let $f_j\to f$ in $W^{1,1}(\mathbb{R})$. Let $x\notin \mathcal{P}$. Assume that $M_2f_j(x)=\intav{[x-r_{j,x},x+r_{j,x}]}|f_j|$ for some $r_{x,j}\ge d(x,\mathcal{P})$. If $r_{j,x}\to r$ then $$M_2f(x)=\intav{[x-r,x+r]}|f|.$$
\end{lemma}
\begin{proof}
This follows as \cite[Lemma 12]{CMP2017}.
\end{proof}
Now we can conclude the pointwise a.e convergence at the derivative level.
\begin{lemma}\label{pointwiseconvergencederivativelevel}
Let $f_j\to f$ in $W^{1,1}(\mathbb{R})$. Then, for $a.e$ $x\in U_{\delta,K}$, we have $$(M_2f_j)'(x)\to (M_2f)'(x).$$
\end{lemma}
\begin{proof}
Let us assume that $x$ is such that $M_{2}f_j$, for every $j$, and $M_{2}f$ are differentiable at the point $x$ and $x\in \left(\frac{a_i+a_{i+1}}{2},a_{i+1}\right)$ for some $i$. The other case follows analogously. Now, for every $j$, let us take $r_{x,j}\ge d(x,\mathcal{P})$ such that $M_2f_j(x)=\intav{[x-r_{j,x},x+r_{j,x}]}|f_j|$. By Lemma \ref{formuladerivativecnentered} we have that $$(M_2f_j)'(x)=\frac{\int_{[x-r_{x,j},x+r_{x,j}]}|f_j|}{2r_{x,j}^2}-\frac{|f_j|(x-r_{x,j})}{r_{x,j}}.$$ 
Assume that there exists a subsequence $\{j_k\}_{k\in \mathbb{N}}$ such that $|(M_2f_{j_k})'(x)-(M_2f)'(x)|>\rho>0$. Let us take $R>0$ such that $\int_{[x-R,x+R]}|f|>\frac{\|f\|_1}{2}$. For $j$ big enough we have that $\int_{[x-R,x+R]}|f_j|>\frac{\|f_j\|_1}{2}$. Since $$\frac{\|f_j\|_1}{4R}<\frac{\int_{[x-R,x+R]}|f_j|}{2R}\le \intav{[x-r_{x,j},x+r_{x,j}]}|f_j|\le \frac{\|f_j\|_1}{2r_{x,j}},$$
we note that $r_{x,j}\le 2R.$ Therefore, there exists a subsequence of $\{j_k\}_{k\in \mathbb{N}}$ (that we keep calling $\{j_k\}_{k\in \mathbb{N}}$ with a harmless abuse of notation) such that $r_{x,j_k}\to r>0$. Thus, by Lemma \ref{convergenceofradiuscentered}, we have \begin{align*}
(M_2f_j)'(x)&=\frac{\int_{[x-r_{x,j},x+r_{x,j}]}|f_j|}{2r_{x,j}^2}-\frac{|f_j|(x-r_{x,j})}{r_{x,j}}\\
&\to \frac{\int_{[x-r,x+r]}|f|}{2r^2}-\frac{|f|(x-r)}{r}=(M_2f)'(x).
\end{align*} 
From this we conclude our lemma.
\end{proof}
We are now in position to conclude our desired $L^{1}(U_{\delta,K})$ convergence.
\begin{proposition}\label{integralconvergencem2}
We have $(M_2f_j)'\to (M_2f)'$ in $L^1(U_{\delta,K}).$
\end{proposition}
\begin{proof}
Let us take $x\in U_{\delta,K}$ with $x\in \left(\frac{a_i+a_{i+1}}{2},a_{i+1}\right)$ and such that $M_2f_j$, for every $j$, and $M_2f$ are all differentiable at the point $x$. The symmetric case follows similarly. By Lemma \ref{formuladerivativecnentered} we have that (using the notation of the previous lemma) 
\begin{align*}\left|(M_2f_j)'(x)\right|&=\left|\frac{\int_{[x-r_{x,j},x+r_{x,j}]}|f_j|}{2r_{x,j}^2}-\frac{|f_j|(x-r_{x,j})}{r_{x,j}}\right|\\
&\le \frac{\|f_j\|_1}{2\delta^2}+\frac{\|f_j\|_{\infty}}{\delta}\le 2\|f\|_{1,1}\left(\frac{1}{2\delta^2}+\frac{1}{\delta}\right),
\end{align*}
for $j$ big enough. Therefore, by combining the dominated convergence theorem with Lemma \ref{pointwiseconvergencederivativelevel}, we conclude our proposition.  
\end{proof}
\subsection{Properties of $M_1$}
About our local operator $M_1$, by Lemma \ref{kurkaslemma} we have that $M_1$ is weakly differentiable in $\mathbb{R}\setminus{\mathcal{P}}$. We now prove the following. 
\begin{proposition}\label{controlofm1}
Let $f_j\to f$ in $W^{1,1}(\mathbb{R})$ (recall that we assume $f_j, f\ge 0$). We have that, for $j$ big enough $$\|(M_1f_j)'-f_j'\|_1\le 2(C+1)\varepsilon,$$ where $C$ is the universal constant appearing in Lemma \ref{kurkaslemma}.
\end{proposition}
\begin{proof}
Let $L_i:(a_i,a_{i+1})\to \mathbb{R}$ be a line such that $L_i'=\alpha_i$ and $L_i\le 0$ (since $\alpha_0=\alpha_n=0$, $L_0$ and $L_n$ are constant). We observe that
\begin{align}\label{firstinequalitycentered}
\int_{(a_i,a_{i+1})}|(M_1f_j)'-(f_j)'|\le \int_{(a_i,a_{i+1})}|(M_1f_j)'-L_i'|+|L_i'-(f_j)'|.     
\end{align}
Let us notice that, for every $x\in (a_i,a_{i+1})$, we have \begin{align*}
M_1f_j-L_i&=\left(\sup_{r<d(x,\{a_i,a_{i+1}\})}\intav{[x-r,x+r]}f_j\right)-L_i\\
&=\sup_{r<d(x,\{a_i,a_{i+1}\})}\intav{[x-r,x+r]}(f_j-L_i)=M_{1}(f_j-L_i).
\end{align*}
Therefore, we have
\begin{align*}
\int_{(a_i,a_{i+1})}|(M_1f_j)'-L_i'|&=\int_{(a_i,a_{i+1})}|(M_1f_j-L_i)'|\\
&=\int_{(a_i,a_{i+1})}|(M_1(f_j-L_i))'|\\
&=C\int_{(a_i,a_{i+1})}|(f_j-L_i)'|.
\end{align*}
Combining this with \eqref{firstinequalitycentered} we have that 
\begin{align*}
 \int_{(a_i,a_{i+1})}|(M_1f_j)'-(f_j)'|&\le (C+1)\int_{(a_i,a_{i+1})}|f_j'-\alpha_i|\le (C+1)\left(\int_{(a_i,a_{i+1})}|f'-\alpha_i|+|f'-f_j'|\right).
\end{align*}
Therefore, we have $$\|(M_1f_j)'-f_j'\|_1\le (C+1)\left(\varepsilon+\|f'-f_j'\|_1\right).$$
Since $\|f'-f_j'\|_1<\varepsilon$ for $j$ big enough, we conclude our proposition.
\end{proof}
Analogously, we conclude that $\|(M_1f)'-(f)'\|_1\le 2(C+1)\varepsilon$, and therefore $\|(M_1f_j)'-(M_1f)'\|_1\le (4C+5)\varepsilon,$ for $j$ big enough.

\begin{figure}[ht]
\includegraphics[scale = 0.8]{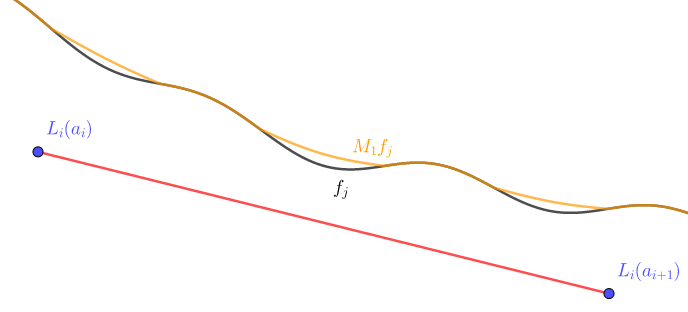}
\caption{$f_j$ and $M_1f_j$ are close at the derivative level to $L_i$ when $j$ is big enough.}
\label{adsf}
\end{figure}
\section{Proof of Theorem \ref{theoremcentered}}
Now we are able to conclude our result.
\begin{proof}
By choosing $K$ big enough and $\delta$ small enough such that Lemmas \ref{controloverfinitenumerofpoints} and \ref{controlnearinfinitycc} hold, we have that
\begin{align}\label{controltrivialpartscc}
\int_{\mathbb{R}\setminus U_{\delta,K}}|(Mf_j)'-(Mf)'|<2\varepsilon,
\end{align}
for $j$ big enough. Now we focus on $U_{\delta,K}$. We follow a similar strategy than in \cite[Lemma 11]{CMP2017}. We observe that $M=\max \{M_1,M_2\}$. Let us write $X_j:=\{x\in U_{\delta,K};M_1f_j(x)>M_2f_j(x)\}$, $Y_j:=\{x\in U_{\delta,K};M_1f_j(x)=M_2f_j(x)\}$ and $Z_j:=\{x\in U_{\delta,K};M_1f_j(x)<M_2f_j(x)\}$. We define $X,Y$ and $Z$ analogously, but this time with respect to $f$ instead of $f_j$. We observe that $(Mf_j)'=(M_1f_j)'$ a.e. in $X_j$, $(Mf_j)'=(M_1f_j)'=(M_2f_j)'$ a.e in $Y_{j}$ and $(M_2f_j)'(x)=(Mf_j)'(x)$ in $Z_j$. Analogous properties hold for $f$ in $X,Y$ and $Z$. Let us observe that 
\begin{align*}
\int_{X}|(Mf_j)'-(Mf)'|&=\int_{X\cap X_j}|(Mf_j)'-(Mf)'|+\int_{X\cap Y_j}|(Mf_j)'-(Mf)'|+\int_{X\cap Z_j}|(Mf_j)'-(Mf)'|\\
&\le \int_{X\cap X_j}|(M_1f_j)'-(M_1f)'|+\int_{X\cap Y_j}|(M_1f_j)'-(M_1f)'|+\int_{X\cap Z_j}|(M_2f_j)'-(M_1f)'|\\
&\le \int_{U_{\delta,K}}|(M_1f_j)'-(M_1f)'|+\int_{X\cap Z_j}|(M_2f_j)'-(M_2f)'|+\int_{X\cap Z_j}|(M_2f)'-(M_1f)'|.
\end{align*}
By Lemma \ref{uniformconvergencecentered} we have that $\chi_{X\cap Z_j}\to 0$ a.e., therefore by the dominated convergence theorem we have $\int_{X\cap Z_j}|(M_2f)'-(M_1f)'|<\varepsilon$ for $j$ big enough. Then, by combining Propositions  \ref{integralconvergencem2} and \ref{controlofm1} with this we have that there exists and universal constant $\tilde{C}$ such that $$\int_{X}|(Mf_j)'-(Mf)'|<\tilde{C}\varepsilon,$$
for $j$ big enough. Similarly, we conclude an analogous statement about $Y$ and $Z$. Therefore, considering \eqref{controltrivialpartscc}, we have that there exist an universal constant $\tilde{\tilde{C}}$ such that 
$$\|(Mf)'-(Mf_j)'\|_1<\tilde{\tilde{C}}\varepsilon,$$ for $j$ big enough. From this we conclude our result.
\end{proof}
\section{Acknowledgements}
The author was supported by CAPES-Brazil and by STEP programme of ICTP-Italy. The author is thankful to Emanuel Carneiro for helpful discussions during the preparation of this manuscript. 

\bibliography{Reference}

\providecommand{\bysame}{\leavevmode\hbox to3em{\hrulefill}\thinspace}
\providecommand{\MR}{\relax\ifhmode\unskip\space\fi MR }
\providecommand{\MRhref}[2]{%
  \href{http://www.ams.org/mathscinet-getitem?mr=#1}{#2}
}
\providecommand{\href}[2]{#2}
\begin{thebibliography}{10}

\bibitem{AP2007}
J.~M. Aldaz and J.~P\'erez~L\'azaro, \emph{Functions of bounded variation, the
  derivative of the one dimensional maximal function, and applications to
  inequalities}, Trans. Amer. Math. Soc. \textbf{359} (2007), no.~5,
  2443--2461. \MR{2276629}

\bibitem{beltran2021continuity}
D.~Beltran, C.~González-Riquelme, J.~Madrid, and J.~Weigt, \emph{Continuity of
  the gradient of the fractional maximal operator on
  ${W}^{1,1}(\mathbb{R}^d)$}, preprint, arxiv.org/abs/2102.10206, 2021.

\bibitem{BM2019}
D.~Beltran and J.~Madrid, \emph{{Endpoint Sobolev continuity of the fractional
  maximal function in higher dimensions}}, Int. Math. Res. Not. (2019), rnz281.

\bibitem{BM2019.2}
D.~Beltran and J.~Madrid, \emph{Regularity of the centered fractional maximal
  function on radial functions}, J. Funct. Anal. \textbf{279} (2020), no.~8,
  108686.

\bibitem{CFS2015}
E.~Carneiro, R.~Finder, and M.~Sousa, \emph{On the variation of maximal
  operators of convolution type {II}}, Rev. Mat. Iberoam. \textbf{34} (2018),
  no.~2, 739--766. \MR{3809456}

\bibitem{CGRM}
E.~Carneiro, C.~Gonz\'alez-Riquelme, and J.~Madrid., \emph{Sunrise strategy for
  the continuity of maximal operators}, preprint, to appear in {\it J. Anal.
  Math.}, arxiv.org/abs/2008.07810, 2020.

\bibitem{CGR}
E.~Carneiro and C.~González-Riquelme, \emph{Gradient bounds for radial maximal
  functions}, Ann. Fenn. Math. \textbf{46} (2021), no.~1, 495–521.

\bibitem{CH2012}
E.~Carneiro and K.~Hughes, \emph{On the endpoint regularity of discrete maximal
  operators}, Math. Res. Lett. \textbf{19} (2012), no.~6, 1245--1262.
  \MR{3091605}

\bibitem{CM2015}
E.~Carneiro and J.~Madrid, \emph{Derivative bounds for fractional maximal
  functions}, Trans. Amer. Math. Soc. \textbf{369} (2017), no.~6, 4063--4092.
  \MR{3624402}

\bibitem{CMP2017}
E.~Carneiro, J.~Madrid, and L.~B. Pierce, \emph{Endpoint {S}obolev and {BV}
  continuity for maximal operators}, J. Funct. Anal. \textbf{273} (2017),
  no.~10, 3262--3294. \MR{3695894}

\bibitem{CS2013}
E.~Carneiro and B.~F. Svaiter, \emph{On the variation of maximal operators of
  convolution type}, J. Funct. Anal. \textbf{265} (2013), no.~5, 837--865.
  \MR{3063097}

\bibitem{GR}
C.~Gonz\'{a}lez-Riquelme, \emph{Sobolev regularity of polar fractional maximal
  functions}, Nonlinear Anal. \textbf{198} (2020), 111889.

\bibitem{gonzalezriquelme2021continuity}
C.~González-Riquelme, \emph{On the continuity of maximal operators of
  convolution type at the derivative level}, preprint, to appear in {\it Israel
  J. Math}, arxiv.org/abs/2105.15020, 2021.

\bibitem{GRK}
C.~González-Riquelme and D.~Kosz, \emph{{BV} continuity for the uncentered
  {H}ardy–{L}ittlewood maximal operator}, J. Funct. Anal. \textbf{281}
  (2021), no.~2, 109037.

\bibitem{HO2004}
P.~Hajlasz and J.~Onninen, \emph{On boundedness of maximal functions in
  {S}obolev spaces}, Ann. Acad. Sci. Fenn. Math. \textbf{29} (2004), no.~1,
  167--176. \MR{2041705}

\bibitem{Kinnunen1997}
J.~Kinnunen, \emph{The {H}ardy-{L}ittlewood maximal function of a {S}obolev
  function}, Israel J. Math. \textbf{100} (1997), 117--124. \MR{1469106}

\bibitem{Kurka2010}
O.~Kurka, \emph{On the variation of the {H}ardy-{L}ittlewood maximal function},
  Ann. Acad. Sci. Fenn. Math. \textbf{40} (2015), no.~1, 109--133. \MR{3310075}

\bibitem{Luiro2007}
H.~Luiro, \emph{Continuity of the maximal operator in {S}obolev spaces}, Proc.
  Amer. Math. Soc. \textbf{135} (2007), no.~1, 243--251. \MR{2280193}

\bibitem{Luiro2018}
\bysame, \emph{On the continuous and discontinuous maximal operators},
  Nonlinear Anal. \textbf{172} (2018), 36--58. \MR{3790366}

\bibitem{Madrid2017}
J.~Madrid, \emph{Endpoint {S}obolev and {BV} continuity for maximal operators,
  {II}}, Rev. Mat. Iberoam. \textbf{35} (2019), no.~7, 2151--2168.

\bibitem{Tanaka2002}
H.~Tanaka, \emph{A remark on the derivative of the one-dimensional
  {H}ardy-{L}ittlewood maximal function}, Bull. Austral. Math. Soc. \textbf{65}
  (2002), no.~2, 253--258. \MR{1898539}

\bibitem{weigtcharacteristic}
J.~Weigt, \emph{Variation of the uncentered maximal characteristic function},
  preprint, arxiv.org/abs/2004.10485, 2020.

\bibitem{weigt2020sobolev}
\bysame, \emph{Endpoint {S}obolev bounds for fractional {H}ardy-{L}ittlewood
  maximal operators}, preprint, arxiv.org/abs/2010.05561, 2021.

\end{thebibliography}
\bibliographystyle{amsplain} 
\end{document}